\documentclass[12pt,a4paper]{amsart}
\usepackage{amssymb,amsthm,amsmath,mathtools,mathabx}
\usepackage{tikz}
\usepackage{url,hyperref}

\usepackage{color}

\newtheorem{theorem}{Theorem}

\newtheorem{lemma}[theorem]{Lemma}

\theoremstyle{definition}

\newcommand\eqdef{\stackrel{\mathrm{def}}{=}}

\title{Generating functions of lattice paths}

\author[R.~Duarte]{Rui Duarte}
\address{CIDMA and Department of Mathematics, University of Aveiro, 3810-193 Aveiro, Portugal}
\email{rduarte@ua.pt}

\author[A.~Guedes~de~Oliveira]{Ant\'onio Guedes de Oliveira}
\address{CMUP and Department of Mathematics, Faculty of Sciences, University of Porto, 4169-007 Porto, Portugal}
\email{agoliv@fc.up.pt}

\thanks{The authors were partially supported by CIDMA and CMUP, respectively, which are financed by national funds through Funda\c{c}\~ao para a Ci\^encia e a Tecnologia (FCT) within projects UIDB/04106/2020 (CIDMA, \url{https://doi.org/10.54499/UIDB/04106/2020}) and UIDB/00144/2020 (CMUP, \url{https://doi.org/10.54499/UIDB/00144/2020})}

\newcommand\N{\mathbb{N}}
\makeatletter
\newcommand\Dotfill {\leavevmode \xleaders \hb@xt@ .55em{\hss .\hss }\hfill \kern \z@}
\makeatother
\newcommand\tri[2]{\bigg(\binom{#1}{#2}\bigg)_{0\leq d\leq n}}

\begin{document}
\begin{abstract}
We recall the main types of \emph{lattice paths}, which are sequences in the lattice of integer coordinates points in the plane. We start with the fundamental \emph{central lattice paths} and \emph{Dyck paths} and proceed in elementary terms through recently introduced lattice paths.

For every type, we consider the respective \emph{generating function}. In fact, through our approach (via \emph{Riordan arrays}), various entries of the On-Line Encyclopedia of Integer Sequences are unified, clarified, and simplified.
\end{abstract}

\maketitle

\section{Introduction}
A \emph{lattice path} is a sequence of points of integer coordinates in the plane such that the difference in coordinates of two consecutive points belongs to a given (small) set of vectors.  In our case, all paths start at $(0,0)$ and end at a point of the $x$-axis. A \emph{central lattice path} of length $2n$ for some $n\in\N$ is a path that starts at $(0,0)$ and ends at $(2n,0)$, such that two subsequent points in the sequence either differ by $D=(1,-1)$ (a down step) or by $U=(1,1)$ (an up step). The path may be seen as a sequence of letters $D$ and $U$ in {equal} number, $n$, and the position of the $n$ letters $D$ (or of the $n$ letters $U$) within the $2n$ letters determines bijectively the path. Hence, the number of central lattice paths of length $2n$ is $\binom{2n}{n}$. This forms a sequence that we can find in the On-Line Encyclopedia of Integer Sequences (OEIS), with reference 
\begin{gather*}\text{\href{https://oeis.org/A000984}{A000984}:}
\bigg(\,\binom{2n}{n}\,\bigg)_{n\geq0}=(1, 2, \textbf{6} \footnotemark, 20, 70, 252, 924, 3\,432, 12\,870,\dotsc)\,.
\shortintertext{Let us now consider the correspondent \emph{generating function}, by definition}
f(x)=\sum_{n\geq0} \binom{2n}{n} x^n\,.
\end{gather*}
\footnotetext{See Figure~2.}

Generating functions were {very successively used} by Euler~\footnote{To count \emph{partitions}, that is, to count the number of ways of writing a positive number as the sum of smaller positive numbers.} with few explanations, in a way that we might think of as ``infinite polynomials'', with operations based on those of polynomials 
---the generating functions associated with quasi-zero sequences---
with similar operational algebraic properties. For example, we define $\Big(\sum_{n\geq0}a_n\,x^n\Big)(0):=a_0$,
\begin{align*}
&\Big(\sum_{n\geq0}a_n\,x^n\Big)+\Big(\sum_{n\geq0}b_n\,x^n\Big):=\sum_{n\geq0}(a_n+b_n)x^n
\shortintertext{and}
&\Big(\sum_{n\geq0}a_n\,x^n\Big)\,\Big(\sum_{n\geq0}b_n\,x^n\Big):=\sum_{n\geq0}\big({\textstyle\sum_{i=0}^n} a_i\,b_{n-i}\big)x^n\,.
\end{align*}
Thus, e.g., 
${\displaystyle\frac{1}{1-x}}=\sum_{n\geq0} x^n$
since $1=\big(\sum_{n\geq0} x^n\big)(1-x)$  in terms of generating functions, or
\[1+ \sum_{n\geq1} 0\cdot x^n=\Big(\sum_{n\geq0} 1\cdot x^n\Big)\Big(1-x +\sum_{n\geq2} 0\cdot x^n\Big)\,.\]
We also write $g(x)=\sqrt{f(x)}$ for generating functions $f(x)$ and $g(x)$ such that $f(0),g(0)\geq0$ with the obvious meaning that $g(x)\,g(x)=f(x)$.
Note that although some generating functions correspond to convergent series, at least in some neighborhoods of 0, others do not, like for example, 
$f(x)=\sum_{n\geq0}n!x^n$.  Yet, in some cases, we may use this correspondence. For example, note that
\begin{align}
&\binom{-1/2}{n}=\frac{(-1/2)(-3/2)\dotsb(-1/2-n+1)}{n!}\nonumber\\
&\hphantom{\binom{-1/2}{n}}=\Big(\!\!-\frac{1}{2}\Big)^n\!\frac{(2n)!}{2^n n!\,n!}\,,\nonumber\\
&\hphantom{\binom{-1/2}{n}}=\Big(\!\!-\frac{1}{4}\Big)^n\binom{2n}{n}\nonumber
\shortintertext{and hence, by the generalized Newton binomial theorem, for $|x|<\frac{1}{4}$, }
&(1-4x)^{-1/2}=\lim_{p\to+\infty}\sum_{n=0}^p\binom{2n}{n}x^n\nonumber
\shortintertext{and so, also \emph{as generating functions}, $\big(\sum_{n\geq0}\binom{2n}{n}\,x^n\big)^2=\frac{1}{1-4x}$ and}
&f(x)=\frac{1}{\sqrt{1-4x}}\,.~\footnotemark\label{eq.fgCLP}
\end{align}
\footnotetext{In particular, since
$\displaystyle\Big(\sum_{n\geq0}\binom{2n}{n}\,x^n\Big)^2=\sum_{n\geq0}4^nx^n$, 
$\displaystyle\sum_{\text{\tiny$\substack{0\leq i,j\leq n\\i+j=n}$}}\binom{2i}{i}\binom{2j}{j}=4^n$.}

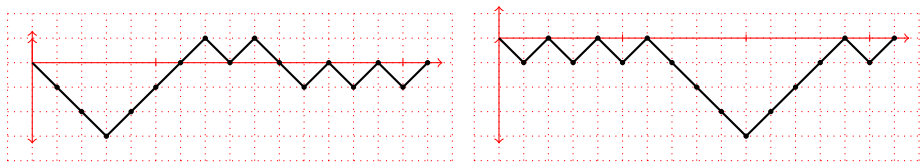
\begin{figure}[ht]
\label{fig1}
\begin{tikzpicture}[scale=.325]
\draw[red,thin,<->] (16.6,0) -- (0,0) -- (0,1.3) ;
\draw[red,thin,<->](0,1) -- (0,-3.3) ;
\draw[step=1,dotted,red] (-1,-4) grid (17,2);
\draw[thin,red] (5,-.15) -- (5,.15);
\draw[thin,red] (10,-.15) -- (10,.15);
\draw[thin,red] (15,-.15) -- (15,.15);
\draw[thick] (0,0)  -- (1,-1) -- (2,-2) -- (3,-3) -- (4,-2) -- (5,-1) -- (6,0) -- (7,1) -- (8,0) -- (9,1) -- (10,0) -- (11,-1) -- (12,0) -- (13,-1) -- (14,0) -- (15,-1) -- (16,0) ;
\draw[fill](1,-1) circle(.085);
\draw[fill](2,-2) circle(.085);
\draw[fill](3,-3) circle(.085);
\draw[fill](4,-2) circle(.085);
\draw[fill](5,-1) circle(.085);
\draw[fill](6,0) circle(.085);
\draw[fill](7,1) circle(.085);
\draw[fill](8,0) circle(.085);
\draw[fill](9,1) circle(.085);
\draw[fill](10,0) circle(.085);
\draw[fill](11,-1) circle(.085);
\draw[fill](12,0) circle(.085);
\draw[fill](13,-1) circle(.085);
\draw[fill](14,0) circle(.085);
\draw[fill](15,-1) circle(.085);
\draw[fill](16,0) circle(.085);
\end{tikzpicture}
\ 
\begin{tikzpicture}[scale=.325]
\draw[red,thin,<->] (16.6,0) -- (0,0) -- (0,1.3) ;
\draw[red,thin,<->](0,0) -- (0,-4.3) ;
\draw[step=1,dotted,red] (-1,-5) grid (17,1);
\draw[thin,red] (5,-.15) -- (5,.15);
\draw[thin,red] (10,-.15) -- (10,.15);
\draw[thin,red] (15,-.15) -- (15,.15);
\draw[thick] (0,0)  -- (1,-1) -- (2,0) -- (3,-1) -- (4,0) -- (5,-1) -- (6,0) -- (7,-1) -- (8,-2) -- (9,-3) -- (10,-4) -- (11,-3) -- (12,-2) -- (13,-1) -- (14,0) -- (15,-1) -- (16,0) ;
\draw[fill](1,-1) circle(.085);
\draw[fill](2,0) circle(.085);
\draw[fill](3,-1) circle(.085);
\draw[fill](4,0) circle(.085);
\draw[fill](5,-1) circle(.085);
\draw[fill](6,0) circle(.085);
\draw[fill](7,-1) circle(.085);
\draw[fill](8,-2) circle(.085);
\draw[fill](9,-3) circle(.085);
\draw[fill](10,-4) circle(.085);
\draw[fill](11,-3) circle(.085);
\draw[fill](12,-2) circle(.085);
\draw[fill](13,-1) circle(.085);
\draw[fill](14,0) circle(.085);
\draw[fill](15,-1) circle(.085);
\draw[fill](16,0) circle(.085);
\end{tikzpicture}
\caption{A \emph{central lattice path} and a \emph{Dyck path}}
\end{figure}

A central lattice path without any point above the $x$-axis is called a \emph{Dyck path}. We may count Dyck paths by counting \emph{non-Dick} central lattice paths $\mathcal{D}$ that \emph{do} have a first point $P$ above the $x$-axis: build a path $\mathcal{D}'$ with the same steps as $\mathcal{D}$ up to $P$ and with the opposite steps afterward. Then $\mathcal{D}'$ is a lattice path from $(0,0)$ to $(2n,2)$, and every lattice path from $(0,0)$ to $(2n,2)$ can be thus obtained. Since there are $\binom{2n}{n-1}$ lattice paths (with $n-1$ down steps and $n+1$ up steps) from $(0,0)$ to $(2n,2)$, the number of Dyck paths is the \emph{Catalan number of order $n$},
\[C_n=\binom{2n}{n}-\binom{2n}{n-1}=\binom{2n}{n}-\frac{n}{n+1}\binom{2n}{n}=\frac{1}{n+1}\binom{2n}{n}\,.\]
This forms the sequence in OEIS
\[\text{\href{https://oeis.org/A000108}{A000108}:}
\big(\,C_n\,\big)_{n\geq0}=(1, 1, \textbf{2}, 5, 14, 42, 132, 429, 1\,430, \dotsc)\,.\]
Since
\[C_n = 2 \binom{2n}{n} - \frac{1}{2} \binom{2(n+1)}{n+1}\,,\]
we have that
\begin{align*}
2x \sum_{n \geq 0} \frac{1}{n+1} \binom{2n}{n} x^n 
& = 4x \sum_{n \geq 0} \binom{2n}{n} x^n - \sum_{n \geq 1} \binom{2n}{n} x^n \\
& = 1 + (4x-1) \sum_{n \geq 0} \binom{2n}{n} x^n \\
& = 1-\sqrt{1-4x} \,.
\end{align*}
Hence,
\begin{equation}
\sum_{n\geq0}C_{n}x^{n}=\frac{2}{1+\sqrt{1-4x}}\,.\label{eq.fgCatalan}
\end{equation}

\section{Other lattice paths}
\subsection{Central Delannoy paths and Schr\"oder paths}
In a \emph{central Delannoy paths} from $(0,0)$ to $(2n,0)$, for $n\in\N$, two subsequent points in the sequence either differ by $D=(1,-1)$ or by $U=(1,1)$, as before, or by a \emph{horizontal step} $H=(2,0)$. The central Delannoy paths that remain below the diagonal are called \emph{Schr\"oder paths}.

Any such path $\mathfrak{d}$ with $d$ down steps must present also $d$ up steps and $n-d$ forward steps. Of course, $\mathfrak{d}$ is determined by the length $2d$ central lattice path $\mathfrak{c}$ formed by the $D$s and the $U$s, and by the positions on $\mathfrak{c}$ where the $n-d$ letters $H$ are placed. Hence, with the same central lattice path $\mathfrak{c}$, there are 
\[\binom{2d+(n-d)}{n-d}=\binom{n+d}{n-d}\]
central Delannoy paths. Note that $\mathfrak{d}$ is a Schr\"oder path if and only if $\mathfrak{c}$ is a Dyck path. The numbers of central Delannoy paths form the OEIS sequence 
\begin{align*}
\text{\href{https://oeis.org/A001850}{A001850}:}&\ \Big(\textstyle\sum\limits_{d=0}^n \binom{n+d}{n-d}\binom{2d}{d}\Big)_{n\geq0}\!\!=
(\text{\small$1, 3, \textbf{13}, 63, 321, 1\,683, 8\,989, 48\,639, \dotsc$})
\shortintertext{whereas the numbers of Schr\"oder paths form the OEIS sequence}
\text{\href{https://oeis.org/A006318}{A006318}:}&\ \Big(\textstyle\sum\limits_{d=0}^n \frac{1}{d+1}\binom{n+d}{n-d}\binom{2d}{d}\Big)_{n\geq0}\!\!=
(\text{\small$1, 2, \textbf{6}, 22, 90, 394, 1\,806, 8\,558, \dotsc$})
\end{align*}

Below, we represent the central Delannoy paths for $n=2$, where the paths drawn in red are Schr\"oder paths. The first six paths are the central lattice paths. 
\begin{figure}[ht]
\label{fig2}
\begin{tikzpicture}[scale=.4]
\draw[thin] (1,-.15) -- (1,.15);
\draw[thin] (2,-.15) -- (2,.15);
\draw[thin] (3,-.15) -- (3,.15);
\draw[thin] (4,-.15) -- (4,.15);
\draw[thin] (-.15,2) -- (.15,2);
\draw[thin] (-.15,1) -- (.15,1);
\draw[thin] (-.15,-1) -- (.15,-1);
\draw[thin] (-.15,-2) -- (.15,-2);
\draw[thick,red] (0,0)  -- (1,-1) -- (2,-2) -- (3,-1) -- (4,0) ;
\draw[fill](1,-1) circle(.085);
\draw[fill](2,-2) circle(.085);
\draw[fill](3,-1) circle(.085);
\draw[fill](4,0) circle(.085);
\draw[thin,<->] (4.3,0) -- (0,0) -- (0,2.3) ;
\draw[thin,<->](0,0) -- (0,-2.3);
\end{tikzpicture}
\
\begin{tikzpicture}[scale=.4]
\draw[thin] (1,-.15) -- (1,.15);
\draw[thin] (2,-.15) -- (2,.15);
\draw[thin] (3,-.15) -- (3,.15);
\draw[thin] (4,-.15) -- (4,.15);
\draw[thin] (-.15,2) -- (.15,2);
\draw[thin] (-.15,1) -- (.15,1);
\draw[thin] (-.15,-1) -- (.15,-1);
\draw[thin] (-.15,-2) -- (.15,-2);
\draw[thick,red] (0,0)  -- (1,-1) -- (2,0) -- (3,-1) -- (4,0) ;
\draw[fill](1,-1) circle(.085);
\draw[fill](2,0) circle(.085);
\draw[fill](3,-1) circle(.085);
\draw[fill](4,0) circle(.085);
\draw[thin,<->] (4.3,0) -- (0,0) -- (0,2.3) ;
\draw[thin,<->](0,0) -- (0,-2.3);
\end{tikzpicture}
\
\begin{tikzpicture}[scale=.4]
\draw[thin] (1,-.15) -- (1,.15);
\draw[thin] (2,-.15) -- (2,.15);
\draw[thin] (3,-.15) -- (3,.15);
\draw[thin] (4,-.15) -- (4,.15);
\draw[thin] (-.15,2) -- (.15,2);
\draw[thin] (-.15,1) -- (.15,1);
\draw[thin] (-.15,-1) -- (.15,-1);
\draw[thin] (-.15,-2) -- (.15,-2);
\draw[thick] (0,0)  -- (1,-1) -- (2,0) -- (3,1) -- (4,0) ;
\draw[fill](1,-1) circle(.085);
\draw[fill](2,0) circle(.085);
\draw[fill](3,1) circle(.085);
\draw[fill](4,0) circle(.085);
\draw[thin,<->] (4.3,0) -- (0,0) -- (0,2.3) ;
\draw[thin,<->](0,0) -- (0,-2.3);
\end{tikzpicture}
\
\begin{tikzpicture}[scale=.4]
\draw[thin] (1,-.15) -- (1,.15);
\draw[thin] (2,-.15) -- (2,.15);
\draw[thin] (3,-.15) -- (3,.15);
\draw[thin] (4,-.15) -- (4,.15);
\draw[thin] (-.15,2) -- (.15,2);
\draw[thin] (-.15,1) -- (.15,1);
\draw[thin] (-.15,-1) -- (.15,-1);
\draw[thin] (-.15,-2) -- (.15,-2);
\draw[thick] (0,0)  -- (1,1) -- (2,0) -- (3,-1) -- (4,0) ;
\draw[fill](1,1) circle(.085);
\draw[fill](2,0) circle(.085);
\draw[fill](3,-1) circle(.085);
\draw[fill](4,0) circle(.085);
\draw[thin,<->] (4.3,0) -- (0,0) -- (0,2.3) ;
\draw[thin,<->](0,0) -- (0,-2.3);
\end{tikzpicture}
\
\begin{tikzpicture}[scale=.4]
\draw[thin] (1,-.15) -- (1,.15);
\draw[thin] (2,-.15) -- (2,.15);
\draw[thin] (3,-.15) -- (3,.15);
\draw[thin] (4,-.15) -- (4,.15);
\draw[thin] (-.15,2) -- (.15,2);
\draw[thin] (-.15,1) -- (.15,1);
\draw[thin] (-.15,-1) -- (.15,-1);
\draw[thin] (-.15,-2) -- (.15,-2);
\draw[thick] (0,0)  -- (1,1) -- (2,0) -- (3,1) -- (4,0) ;
\draw[fill](1,1) circle(.085);
\draw[fill](2,0) circle(.085);
\draw[fill](3,1) circle(.085);
\draw[fill](4,0) circle(.085);
\draw[thin,<->] (4.3,0) -- (0,0) -- (0,2.3) ;
\draw[thin,<->](0,0) -- (0,-2.3);
\end{tikzpicture}
\
\begin{tikzpicture}[scale=.4]
\draw[thin] (1,-.15) -- (1,.15);
\draw[thin] (2,-.15) -- (2,.15);
\draw[thin] (3,-.15) -- (3,.15);
\draw[thin] (4,-.15) -- (4,.15);
\draw[thin] (-.15,2) -- (.15,2);
\draw[thin] (-.15,1) -- (.15,1);
\draw[thin] (-.15,-1) -- (.15,-1);
\draw[thin] (-.15,-2) -- (.15,-2);
\draw[thick] (0,0)  -- (1,1) -- (2,2) -- (3,1) -- (4,0) ;
\draw[fill](1,1) circle(.085);
\draw[fill](2,2) circle(.085);
\draw[fill](3,1) circle(.085);
\draw[fill](4,0) circle(.085);
\draw[thin,<->] (4.3,0) -- (0,0) -- (0,2.3) ;
\draw[thin,<->](0,0) -- (0,-2.3);
\end{tikzpicture}

\begin{tikzpicture}[scale=.4]
\draw[thin] (1,-.15) -- (1,.15);
\draw[thin] (2,-.15) -- (2,.15);
\draw[thin] (3,-.15) -- (3,.15);
\draw[thin] (4,-.15) -- (4,.15);
\draw[thin] (-.15,2) -- (.15,2);
\draw[thin] (-.15,1) -- (.15,1);
\draw[thin] (-.15,-1) -- (.15,-1);
\draw[thin] (-.15,-2) -- (.15,-2);
\draw[thick,red] (0,0)  -- (2,0) -- (2,0) -- (3,-1) -- (4,0) ;
\draw[fill](2,0) circle(.085);
\draw[fill](2,0) circle(.085);
\draw[fill](3,-1) circle(.085);
\draw[fill](4,0) circle(.085);
\draw[thin,<-] (4.3,0) -- (2,0) ;
\draw[thin,->] (0,0) -- (0,2.3) ;
\draw[thin,<->](0,0) -- (0,-2.3);
\end{tikzpicture}
\
\begin{tikzpicture}[scale=.4]
\draw[thin] (1,-.15) -- (1,.15);
\draw[thin] (2,-.15) -- (2,.15);
\draw[thin] (3,-.15) -- (3,.15);
\draw[thin] (4,-.15) -- (4,.15);
\draw[thin] (-.15,2) -- (.15,2);
\draw[thin] (-.15,1) -- (.15,1);
\draw[thin] (-.15,-1) -- (.15,-1);
\draw[thin] (-.15,-2) -- (.15,-2);
\draw[thick] (0,0)  -- (2,0) -- (2,0) -- (3,1) -- (4,0) ;
\draw[fill](2,0) circle(.085);
\draw[fill](2,0) circle(.085);
\draw[fill](3,1) circle(.085);
\draw[fill](4,0) circle(.085);
\draw[thin,<-] (4.3,0) -- (2,0) ;
\draw[thin,->] (0,0) -- (0,2.3) ;
\draw[thin,<->](0,0) -- (0,-2.3);
\end{tikzpicture}
\
\begin{tikzpicture}[scale=.4]
\draw[thin] (1,-.15) -- (1,.15);
\draw[thin] (2,-.15) -- (2,.15);
\draw[thin] (3,-.15) -- (3,.15);
\draw[thin] (4,-.15) -- (4,.15);
\draw[thin] (-.15,2) -- (.15,2);
\draw[thin] (-.15,1) -- (.15,1);
\draw[thin] (-.15,-1) -- (.15,-1);
\draw[thin] (-.15,-2) -- (.15,-2);
\draw[thick,red] (0,0)  -- (1,-1) -- (3,-1) -- (3,-1) -- (4,0) ;
\draw[fill](1,-1) circle(.085);
\draw[fill](3,-1) circle(.085);
\draw[fill](3,-1) circle(.085);
\draw[fill](4,0) circle(.085);
\draw[thin,<->] (4.3,0) -- (0,0) -- (0,2.3) ;
\draw[thin,<->](0,0) -- (0,-2.3);
\end{tikzpicture}
\
\begin{tikzpicture}[scale=.4]
\draw[thin] (1,-.15) -- (1,.15);
\draw[thin] (2,-.15) -- (2,.15);
\draw[thin] (3,-.15) -- (3,.15);
\draw[thin] (4,-.15) -- (4,.15);
\draw[thin] (-.15,2) -- (.15,2);
\draw[thin] (-.15,1) -- (.15,1);
\draw[thin] (-.15,-1) -- (.15,-1);
\draw[thin] (-.15,-2) -- (.15,-2);
\draw[very thick, densely dashed] (0,0)  -- (1,1) -- (3,1) -- (3,1) ;
\draw[thick]  (3,1) -- (4,0) ;
\draw[fill](1,1) circle(.085);
\draw[fill](3,1) circle(.085);
\draw[fill](3,1) circle(.085);
\draw[fill](4,0) circle(.085);
\draw[thin,<->] (4.3,0) -- (0,0) -- (0,2.3) ;
\draw[thin,<->](0,0) -- (0,-2.3);
\end{tikzpicture}
\
\begin{tikzpicture}[scale=.4]
\draw[thin,<-] (4.3,0) -- (4,0)  ;
\draw[thin,->] (2,0) -- (0,0) -- (0,2.3) ;
\draw[thin,<->](0,0) -- (0,-2.3);
\draw[thin] (1,-.15) -- (1,.15);
\draw[thin] (2,-.15) -- (2,.15);
\draw[thin] (3,-.15) -- (3,.15);
\draw[thin] (4,-.15) -- (4,.15);
\draw[thin] (-.15,2) -- (.15,2);
\draw[thin] (-.15,1) -- (.15,1);
\draw[thin] (-.15,-1) -- (.15,-1);
\draw[thin] (-.15,-2) -- (.15,-2);
\draw[thick,red] (0,0)  -- (1,-1) ;
\draw[very thick,red, densely dashed] (1,-1) -- (2,0) -- (4,0) -- (4,0) ;
\draw[fill](1,-1) circle(.085);
\draw[fill](2,0) circle(.085);
\draw[fill](4,0) circle(.085);
\draw[fill](4,0) circle(.085);
\end{tikzpicture}
\
\begin{tikzpicture}[scale=.4]
\draw[thin] (1,-.15) -- (1,.15);
\draw[thin] (2,-.15) -- (2,.15);
\draw[thin] (3,-.15) -- (3,.15);
\draw[thin] (4,-.15) -- (4,.15);
\draw[thin] (-.15,2) -- (.15,2);
\draw[thin] (-.15,1) -- (.15,1);
\draw[thin] (-.15,-1) -- (.15,-1);
\draw[thin] (-.15,-2) -- (.15,-2);
\draw[thick] (0,0)  -- (1,1) -- (2,0) -- (4,0) -- (4,0) ;
\draw[fill](1,1) circle(.085);
\draw[fill](2,0) circle(.085);
\draw[fill](4,0) circle(.085);
\draw[fill](4,0) circle(.085);
\draw[thin,<-] (4.3,0) -- (4,0)  ;
\draw[thin,->] (2,0) -- (0,0) -- (0,2.3) ;
\draw[thin,<->](0,0) -- (0,-2.3);
\end{tikzpicture}

\begin{tikzpicture}[scale=.4]
\draw[thin,<->] (4.3,0) -- (4,0) ;
\draw[thin,<->] (0,0) -- (0,2.3) ;
\draw[thin,<->](0,0) -- (0,-2.3);
\draw[thin] (1,-.15) -- (1,.15);
\draw[thin] (2,-.15) -- (2,.15);
\draw[thin] (3,-.15) -- (3,.15);
\draw[thin] (4,-.15) -- (4,.15);
\draw[thin] (-.15,2) -- (.15,2);
\draw[thin] (-.15,1) -- (.15,1);
\draw[thin] (-.15,-1) -- (.15,-1);
\draw[thin] (-.15,-2) -- (.15,-2);
\draw[thick,red] (0,0)  -- (2,0) -- (2,0) -- (4,0) -- (4,0) ;
\draw[fill](2,0) circle(.085);
\draw[fill](2,0) circle(.085);
\draw[fill](4,0) circle(.085);
\draw[fill](4,0) circle(.085);
\end{tikzpicture}
\caption{Delannoy and Schr\"oder paths with $n=2$}
\end{figure}
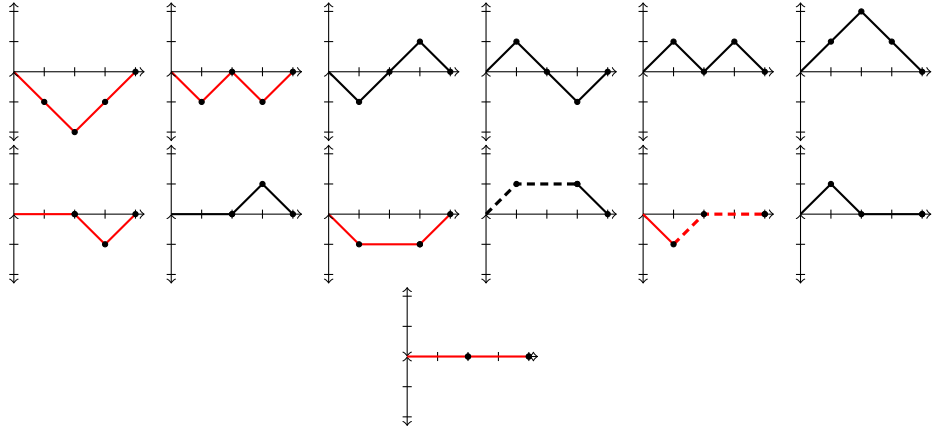

\subsection{[Big] Motzkin paths}
If, instead of allowing steps with $H=(2,0)$, we allow steps with $F=(1,0)$, the path $\mathfrak{m}$ is called a \emph{big Motzkin path}. A big Motzkin path is simply a \emph{Motzkin path} if the path remains below the $x$-axis. Note that, if the path contains $d$ down steps, then it contains also $d$ up steps and $n-2d$ forward steps. Hence, for a given central lattice path $\mathfrak{c}$, there are 
\[\binom{2d+(n-2d)}{n-2d}=\binom{n}{n-2d}=\binom{n}{2d}\]
Motzkin paths. Again, Motzkin paths occur when $\mathfrak{c}$ is a Dyck path. The numbers of big Motzkin paths form the \emph{sequence of central trinomial coefficients}, the OEIS sequence
\begin{align*}
\text{\href{https://oeis.org/A002426}{A002426}:}&\ \Big(\textstyle\sum\limits_{d=0}^n \binom{n}{2d}\binom{2d}{d}\Big)_{n\geq0}\!\!=
(\text{\small$1, 1, 3, 7, 19, \textbf{51}, 141, 393, 1\,107, 3\,139,\dotsc$})
\shortintertext{whereas the numbers of Motzkin paths form the OEIS sequence} 
\text{\href{https://oeis.org/A001006}{A001006}:}&\ \Big(\textstyle\sum\limits_{d=0}^n \frac{1}{d+1}\binom{n}{2d}\binom{2d}{d}\Big)_{n\geq0}\!\!=
(\text{\small$1, 1, 2, 4, 9, \textbf{21}, 51, 127, 323, 835,\dotsc$})
\end{align*}

We first represent the $21$ Motzkin paths and then the remaining $30$ big Motzkin paths of length $5$.
\begin{figure}[ht]
\label{fig3}
\begin{align*}
&\begin{tikzpicture}[scale=.3]
\draw[thin,<->] (5.3,0) -- (0,0) -- (0,2.5) ;
\draw[thin,->] (0,0) -- (0,-2.5) ;
\draw (0,0)  -- (1,0) -- (2,0) -- (3,0) -- (4,0) -- (5,0) ;
\draw[fill](1,0) circle(.1);
\draw[fill](2,0) circle(.1);
\draw[fill](3,0) circle(.1);
\draw[fill](4,0) circle(.1);
\draw[fill](5,0) circle(.1);
\end{tikzpicture}
\,\begin{tikzpicture}[scale=.3]
\draw[thin,<-] (5.3,0) -- (3,0)  ;
\draw[thin,->] (2,0) -- (0,0) -- (0,2.5) ;
\draw[thin,->] (0,0) -- (0,-2.5) ;
\draw (0,0)  -- (1,-1) ;
\draw[red, very thick, densely dashed] (1,-1) -- (2,0) -- (3,0) ;
\draw (3,0) -- (4,0) -- (5,0) ;
\draw[fill](1,-1) circle(.1);
\draw[fill](2,0) circle(.1);
\draw[fill](3,0) circle(.1);
\draw[fill](4,0) circle(.1);
\draw[fill](5,0) circle(.1);
\end{tikzpicture}
\,\begin{tikzpicture}[scale=.3]
\draw[thin,<-] (5.3,0) -- (4,0) ;
\draw[thin,->] (3,0) -- (0,0) -- (0,2.5) ;
\draw[thin,->] (0,0) -- (0,-2.5) ;
\draw (0,0)  -- (1,-1) -- (2,-1)  ;
\draw[red, very thick, densely dashed] (2,-1) -- (3,0) -- (4,0) ;
\draw (4,0) -- (5,0) ;
\draw[fill](1,-1) circle(.1);
\draw[fill](2,-1) circle(.1);
\draw[fill](3,0) circle(.1);
\draw[fill](4,0) circle(.1);
\draw[fill](5,0) circle(.1);
\end{tikzpicture}
\,\begin{tikzpicture}[scale=.3]
\draw[thin,<-] (5.3,0) -- (5,0) ;
\draw[thin,->] (4,0) -- (0,0) -- (0,2.5) ;
\draw[thin,->] (0,0) -- (0,-2.5) ;
\draw (0,0)  -- (1,-1) -- (2,-1) -- (3,-1) ;
\draw[red, very thick, densely dashed] (3,-1) -- (4,0) -- (5,0) ;
\draw[fill](1,-1) circle(.1);
\draw[fill](2,-1) circle(.1);
\draw[fill](3,-1) circle(.1);
\draw[fill](4,0) circle(.1);
\draw[fill](5,0) circle(.1);
\end{tikzpicture}
\,\begin{tikzpicture}[scale=.3]
\draw[thin,<->] (5.3,0) -- (0,0) -- (0,2.5) ;
\draw[thin,->] (0,0) -- (0,-2.5) ;
\draw (0,0)  -- (1,-1) -- (2,-1) -- (3,-1) -- (4,-1) -- (5,0) ;
\draw[fill](1,-1) circle(.1);
\draw[fill](2,-1) circle(.1);
\draw[fill](3,-1) circle(.1);
\draw[fill](4,-1) circle(.1);
\draw[fill](5,0) circle(.1);
\end{tikzpicture}
\,\begin{tikzpicture}[scale=.3]
\draw[thin,<-] (5.3,0) -- (4,0) ;
\draw[thin,->] (3,0) -- (0,0) -- (0,2.5) ;
\draw[thin,->] (0,0) -- (0,-2.5) ;
\draw (0,0)  -- (1,0) -- (2,-1) ;
\draw[red, very thick, densely dashed] (2,-1) -- (3,0) -- (4,0)  ;
\draw (4,0) -- (5,0) ;
\draw[fill](1,0) circle(.1);
\draw[fill](2,-1) circle(.1);
\draw[fill](3,0) circle(.1);
\draw[fill](4,0) circle(.1);
\draw[fill](5,0) circle(.1);
\end{tikzpicture}
\,\begin{tikzpicture}[scale=.3]
\draw[thin,<-] (5.3,0) -- (5,0) ;
\draw[thin,->] (4,0) -- (0,0) -- (0,2.5) ;
\draw[thin,->] (0,0) -- (0,-2.5) ;
\draw (0,0)  -- (1,0) -- (2,-1) -- (3,-1) ;
\draw[red, very thick, densely dashed] (3,-1) -- (4,0) -- (5,0) ;
\draw[fill](1,0) circle(.1);
\draw[fill](2,-1) circle(.1);
\draw[fill](3,-1) circle(.1);
\draw[fill](4,0) circle(.1);
\draw[fill](5,0) circle(.1);
\end{tikzpicture}\\[-2.5pt]
&\begin{tikzpicture}[scale=.3]
\draw[thin,<->] (5.3,0) -- (0,0) -- (0,2.5) ;
\draw[thin,->] (0,0) -- (0,-2.5) ;
\draw (0,0)  -- (1,0) -- (2,-1) -- (3,-1) -- (4,-1) -- (5,0) ;
\draw[fill](1,0) circle(.1);
\draw[fill](2,-1) circle(.1);
\draw[fill](3,-1) circle(.1);
\draw[fill](4,-1) circle(.1);
\draw[fill](5,0) circle(.1);
\end{tikzpicture}
\,\begin{tikzpicture}[scale=.3]
\draw[thin,<-] (5.3,0) -- (5,0) ;
\draw[thin,->] (4,0) -- (0,0) -- (0,2.5) ;
\draw[thin,->] (0,0) -- (0,-2.5) ;
\draw (0,0)  -- (1,0) -- (2,0) -- (3,-1) ;
\draw[red, very thick, densely dashed]  (3,-1) -- (4,0) -- (5,0) ;
\draw[fill](1,0) circle(.1);
\draw[fill](2,0) circle(.1);
\draw[fill](3,-1) circle(.1);
\draw[fill](4,0) circle(.1);
\draw[fill](5,0) circle(.1);
\end{tikzpicture}
\,\begin{tikzpicture}[scale=.3]
\draw[thin,<->] (5.3,0) -- (0,0) -- (0,2.5) ;
\draw[thin,->] (0,0) -- (0,-2.5) ;
\draw (0,0)  -- (1,0) -- (2,0) -- (3,-1) -- (4,-1) -- (5,0) ;
\draw[fill](1,0) circle(.1);
\draw[fill](2,0) circle(.1);
\draw[fill](3,-1) circle(.1);
\draw[fill](4,-1) circle(.1);
\draw[fill](5,0) circle(.1);
\end{tikzpicture}
\,\begin{tikzpicture}[scale=.3]
\draw[thin,<->] (5.3,0) -- (0,0) -- (0,2.5) ;
\draw[thin,->] (0,0) -- (0,-2.5) ;
\draw (0,0)  -- (1,0) -- (2,0) -- (3,0) -- (4,-1) -- (5,0) ;
\draw[fill](1,0) circle(.1);
\draw[fill](2,0) circle(.1);
\draw[fill](3,0) circle(.1);
\draw[fill](4,-1) circle(.1);
\draw[fill](5,0) circle(.1);
\end{tikzpicture}
\,\begin{tikzpicture}[scale=.3]
\draw[thin,<-] (5.3,0) -- (5,0) ;
\draw[thin,->] (4,0) -- (0,0) -- (0,2.5) ;
\draw[thin,->] (0,0) -- (0,-2.5) ;
\draw (0,0)  -- (1,-1) -- (2,-2) -- (3,-1) ;
\draw[red, very thick, densely dashed] (3,-1) -- (4,0) -- (5,0) ;
\draw[fill](1,-1) circle(.1);
\draw[fill](2,-2) circle(.1);
\draw[fill](3,-1) circle(.1);
\draw[fill](4,0) circle(.1);
\draw[fill](5,0) circle(.1);
\end{tikzpicture}
\,\begin{tikzpicture}[scale=.3]
\draw[thin,<-] (5.3,0) -- (5,0) ;
\draw[thin,->] (4,0) -- (0,0) -- (0,2.5) ;
\draw[thin,->] (0,0) -- (0,-2.5) ;
\draw (0,0)  -- (1,-1) -- (2,0) -- (3,-1) ;
\draw[red, very thick, densely dashed] (3,-1) -- (4,0) -- (5,0) ;
\draw[fill](1,-1) circle(.1);
\draw[fill](2,0) circle(.1);
\draw[fill](3,-1) circle(.1);
\draw[fill](4,0) circle(.1);
\draw[fill](5,0) circle(.1);
\end{tikzpicture}
\,\begin{tikzpicture}[scale=.3]
\draw[thin,<->] (5.3,0) -- (0,0) -- (0,2.5) ;
\draw[thin,->] (0,0) -- (0,-2.5) ;
\draw (0,0)  -- (1,-1) -- (2,-2) ;
\draw[red, very thick, densely dashed]  (2,-2) -- (3,-1) -- (4,-1) ;
\draw (4,-1) -- (5,0) ;
\draw[fill](1,-1) circle(.1);
\draw[fill](2,-2) circle(.1);
\draw[fill](3,-1) circle(.1);
\draw[fill](4,-1) circle(.1);
\draw[fill](5,0) circle(.1);
\end{tikzpicture}\\[-2.5pt]
&\begin{tikzpicture}[scale=.3]
\draw[thin,<->] (5.3,0) -- (0,0) -- (0,2.5) ;
\draw[thin,->] (0,0) -- (0,-2.5) ;
\draw (0,0)  -- (1,-1) -- (2,0) -- (3,-1) -- (4,-1) -- (5,0) ;
\draw[fill](1,-1) circle(.1);
\draw[fill](2,0) circle(.1);
\draw[fill](3,-1) circle(.1);
\draw[fill](4,-1) circle(.1);
\draw[fill](5,0) circle(.1);
\end{tikzpicture}
\,\begin{tikzpicture}[scale=.3]
\draw[thin,<->] (5.3,0) -- (0,0) -- (0,2.5) ;
\draw[thin,->] (0,0) -- (0,-2.5) ;
\draw (0,0)  -- (1,-1) -- (2,-2) -- (3,-2) -- (4,-1) -- (5,0) ;
\draw[fill](1,-1) circle(.1);
\draw[fill](2,-2) circle(.1);
\draw[fill](3,-2) circle(.1);
\draw[fill](4,-1) circle(.1);
\draw[fill](5,0) circle(.1);
\end{tikzpicture}
\,\begin{tikzpicture}[scale=.3]
\draw[thin,<-] (5.3,0) -- (3,0) ;
\draw[thin,->] (2,0) -- (0,0) -- (0,2.5) ;
\draw[thin,->] (0,0) -- (0,-2.5) ;
\draw (0,0)  -- (1,-1) ;
\draw[red, very thick, densely dashed] (1,-1) -- (2,0) -- (3,0) ;
\draw (3,0) -- (4,-1) -- (5,0) ;
\draw[fill](1,-1) circle(.1);
\draw[fill](2,0) circle(.1);
\draw[fill](3,0) circle(.1);
\draw[fill](4,-1) circle(.1);
\draw[fill](5,0) circle(.1);
\end{tikzpicture}
\,\begin{tikzpicture}[scale=.3]
\draw[thin,<->] (5.3,0) -- (0,0) -- (0,2.5) ;
\draw[thin,->] (0,0) -- (0,-2.5) ;
\draw (0,0)  -- (1,-1) -- (2,-1) -- (3,-2) -- (4,-1) -- (5,0) ;
\draw[fill](1,-1) circle(.1);
\draw[fill](2,-1) circle(.1);
\draw[fill](3,-2) circle(.1);
\draw[fill](4,-1) circle(.1);
\draw[fill](5,0) circle(.1);
\end{tikzpicture}
\,\begin{tikzpicture}[scale=.3]
\draw[thin,<->] (5.3,0) -- (0,0) -- (0,2.5) ;
\draw[thin,->] (0,0) -- (0,-2.5) ;
\draw (0,0)  -- (1,-1) -- (2,-1) -- (3,0) -- (4,-1) -- (5,0) ;
\draw[fill](1,-1) circle(.1);
\draw[fill](2,-1) circle(.1);
\draw[fill](3,0) circle(.1);
\draw[fill](4,-1) circle(.1);
\draw[fill](5,0) circle(.1);
\end{tikzpicture}
\,\begin{tikzpicture}[scale=.3]
\draw[thin,<->] (5.3,0) -- (0,0) -- (0,2.5) ;
\draw[thin,->] (0,0) -- (0,-2.5) ;
\draw (0,0)  -- (1,0) -- (2,-1) -- (3,-2) -- (4,-1) -- (5,0) ;
\draw[fill](1,0) circle(.1);
\draw[fill](2,-1) circle(.1);
\draw[fill](3,-2) circle(.1);
\draw[fill](4,-1) circle(.1);
\draw[fill](5,0) circle(.1);
\end{tikzpicture}
\,\begin{tikzpicture}[scale=.3]
\draw[thin,<->] (5.3,0) -- (0,0) -- (0,2.5) ;
\draw[thin,->] (0,0) -- (0,-2.5) ;
\draw (0,0)  -- (1,0) -- (2,-1) -- (3,0) -- (4,-1) -- (5,0) ;
\draw[fill](1,0) circle(.1);
\draw[fill](2,-1) circle(.1);
\draw[fill](3,0) circle(.1);
\draw[fill](4,-1) circle(.1);
\draw[fill](5,0) circle(.1);
\end{tikzpicture}\\[2.5pt]
\hline\\[-2.5pt]
&\begin{tikzpicture}[scale=.265]
\draw[thin,<->] (5.3,0) -- (0,0) -- (0,2.5) ;
\draw[thin,->] (0,0) -- (0,-2.5) ;
\draw (0,0)  -- (1,1) -- (2,0) -- (3,0) -- (4,0) -- (5,0) ;
\draw[fill](1,1) circle(.1);
\draw[fill](2,0) circle(.1);
\draw[fill](3,0) circle(.1);
\draw[fill](4,0) circle(.1);
\draw[fill](5,0) circle(.1);
\end{tikzpicture}
\,\begin{tikzpicture}[scale=.265]
\draw[thin,<->] (5.3,0) -- (0,0) -- (0,2.5) ;
\draw[thin,->] (0,0) -- (0,-2.5) ;
\draw[red, very thick, densely dashed] (0,0)  -- (1,1) -- (2,1) ;
\draw (2,1) -- (3,0) -- (4,0) -- (5,0) ;
\draw[fill](1,1) circle(.1);
\draw[fill](2,1) circle(.1);
\draw[fill](3,0) circle(.1);
\draw[fill](4,0) circle(.1);
\draw[fill](5,0) circle(.1);
\end{tikzpicture}
\,\begin{tikzpicture}[scale=.265]
\draw[thin,<->] (5.3,0) -- (0,0) -- (0,2.5) ;
\draw[thin,->] (0,0) -- (0,-2.5) ;
\draw[red, very thick, densely dashed] (0,0)  -- (1,1) -- (2,1) ;
\draw (2,1) -- (3,1) -- (4,0) -- (5,0) ;
\draw[fill](1,1) circle(.1);
\draw[fill](2,1) circle(.1);
\draw[fill](3,1) circle(.1);
\draw[fill](4,0) circle(.1);
\draw[fill](5,0) circle(.1);
\end{tikzpicture}
\,\begin{tikzpicture}[scale=.265]
\draw[thin,<->] (5.3,0) -- (0,0) -- (0,2.5) ;
\draw[thin,->] (0,0) -- (0,-2.5) ;
\draw[red, very thick, densely dashed] (0,0)  -- (1,1) -- (2,1) ;
\draw (2,1) -- (3,1) -- (4,1) -- (5,0) ;
\draw[fill](1,1) circle(.1);
\draw[fill](2,1) circle(.1);
\draw[fill](3,1) circle(.1);
\draw[fill](4,1) circle(.1);
\draw[fill](5,0) circle(.1);
\end{tikzpicture}
\,\begin{tikzpicture}[scale=.265]
\draw[thin,<->] (5.3,0) -- (0,0) -- (0,2.5) ;
\draw[thin,->] (0,0) -- (0,-2.5) ;
\draw (0,0)  -- (1,0) -- (2,1) -- (3,0) -- (4,0) -- (5,0) ;
\draw[fill](1,0) circle(.1);
\draw[fill](2,1) circle(.1);
\draw[fill](3,0) circle(.1);
\draw[fill](4,0) circle(.1);
\draw[fill](5,0) circle(.1);
\end{tikzpicture}
\,\begin{tikzpicture}[scale=.265]
\draw[thin,<->] (5.3,0) -- (0,0) -- (0,2.5) ;
\draw[thin,->] (0,0) -- (0,-2.5) ;
\draw (0,0)  -- (1,0) ;
\draw[red, very thick, densely dashed] (1,0) -- (2,1) -- (3,1) ;
\draw (3,1) -- (4,0) -- (5,0) ;
\draw[fill](1,0) circle(.1);
\draw[fill](2,1) circle(.1);
\draw[fill](3,1) circle(.1);
\draw[fill](4,0) circle(.1);
\draw[fill](5,0) circle(.1);
\end{tikzpicture}
\,\begin{tikzpicture}[scale=.265]
\draw[thin,<->] (5.3,0) -- (0,0) -- (0,2.5) ;
\draw[thin,->] (0,0) -- (0,-2.5) ;
\draw (0,0)  -- (1,0) ;
\draw[red, very thick, densely dashed] (1,0) -- (2,1) -- (3,1) ;
\draw (3,1) -- (4,1) -- (5,0) ;
\draw[fill](1,0) circle(.1);
\draw[fill](2,1) circle(.1);
\draw[fill](3,1) circle(.1);
\draw[fill](4,1) circle(.1);
\draw[fill](5,0) circle(.1);
\end{tikzpicture}
\,\begin{tikzpicture}[scale=.265]
\draw[thin,<->] (5.3,0) -- (0,0) -- (0,2.5) ;
\draw[thin,->] (0,0) -- (0,-2.5) ;
\draw (0,0)  -- (1,0) -- (2,0) -- (3,1) -- (4,0) -- (5,0) ;
\draw[fill](1,0) circle(.1);
\draw[fill](2,0) circle(.1);
\draw[fill](3,1) circle(.1);
\draw[fill](4,0) circle(.1);
\draw[fill](5,0) circle(.1);
\end{tikzpicture}\\[-2.5pt]
&\begin{tikzpicture}[scale=.265]
\draw[thin,<-] (5.3,0) -- (2,0) ;
\draw[thin,->] (1,0) -- (0,0) -- (0,2.5) ;
\draw[thin,->] (0,0) -- (0,-2.5) ;
\draw (0,0)  -- (1,0) -- (2,0) ;
\draw[red, very thick, densely dashed]  (2,0) -- (3,1) -- (4,1) ;
\draw (4,1) -- (5,0) ;
\draw[fill](1,0) circle(.1);
\draw[fill](2,0) circle(.1);
\draw[fill](3,1) circle(.1);
\draw[fill](4,1) circle(.1);
\draw[fill](5,0) circle(.1);
\end{tikzpicture}
\,\begin{tikzpicture}[scale=.265]
\draw[thin,<-] (5.3,0) -- (3,0) ;
\draw[thin,->] (2,0) -- (0,0) -- (0,2.5) ;
\draw[thin,->] (0,0) -- (0,-2.5) ;
\draw (0,0)  -- (1,0) -- (2,0) -- (3,0) -- (4,1) -- (5,0) ;
\draw[fill](1,0) circle(.1);
\draw[fill](2,0) circle(.1);
\draw[fill](3,0) circle(.1);
\draw[fill](4,1) circle(.1);
\draw[fill](5,0) circle(.1);
\end{tikzpicture}
\,\begin{tikzpicture}[scale=.265]
\draw[thin,<->] (5.3,0) -- (0,0) -- (0,2.5) ;
\draw[thin,->] (0,0) -- (0,-2.5) ;
\draw (0,0)  -- (1,-1) -- (2,0) -- (3,1) -- (4,0) -- (5,0) ;
\draw[fill](1,-1) circle(.1);
\draw[fill](2,0) circle(.1);
\draw[fill](3,1) circle(.1);
\draw[fill](4,0) circle(.1);
\draw[fill](5,0) circle(.1);
\end{tikzpicture}
\,\begin{tikzpicture}[scale=.265]
\draw[thin,<->] (5.3,0) -- (0,0) -- (0,2.5) ;
\draw[thin,->] (0,0) -- (0,-2.5) ;
\draw (0,0)  -- (1,1) -- (2,0) -- (3,-1) ;
\draw[red, very thick, densely dashed] (3,-1) -- (4,0) -- (5,0) ;
\draw[fill](1,1) circle(.1);
\draw[fill](2,0) circle(.1);
\draw[fill](3,-1) circle(.1);
\draw[fill](4,0) circle(.1);
\draw[fill](5,0) circle(.1);
\end{tikzpicture}
\,\begin{tikzpicture}[scale=.265]
\draw[thin,<->] (5.3,0) -- (0,0) -- (0,2.5) ;
\draw[thin,->] (0,0) -- (0,-2.5) ;
\draw (0,0)  -- (1,1) -- (2,0) -- (3,1) -- (4,0) -- (5,0) ;
\draw[fill](1,1) circle(.1);
\draw[fill](2,0) circle(.1);
\draw[fill](3,1) circle(.1);
\draw[fill](4,0) circle(.1);
\draw[fill](5,0) circle(.1);
\end{tikzpicture}
\,\begin{tikzpicture}[scale=.265]
\draw[thin,<->] (5.3,0) -- (0,0) -- (0,2.5) ;
\draw[thin,->] (0,0) -- (0,-2.5) ;
\draw (0,0)  -- (1,1) -- (2,2) -- (3,1) -- (4,0) -- (5,0) ;
\draw[fill](1,1) circle(.1);
\draw[fill](2,2) circle(.1);
\draw[fill](3,1) circle(.1);
\draw[fill](4,0) circle(.1);
\draw[fill](5,0) circle(.1);
\end{tikzpicture}
\,\begin{tikzpicture}[scale=.265]
\draw[thin,<->] (5.3,0) -- (0,0) -- (0,2.5) ;
\draw[thin,->] (0,0) -- (0,-2.5) ;
\draw (0,0)  -- (1,-1) -- (2,0) ;
\draw[red, very thick, densely dashed] (2,0) -- (3,1) -- (4,1) ;
\draw (4,1) -- (5,0) ;
\draw[fill](1,-1) circle(.1);
\draw[fill](2,0) circle(.1);
\draw[fill](3,1) circle(.1);
\draw[fill](4,1) circle(.1);
\draw[fill](5,0) circle(.1);
\end{tikzpicture}
\,\begin{tikzpicture}[scale=.265]
\draw[thin,<->] (5.3,0) -- (0,0) -- (0,2.5) ;
\draw[thin,->] (0,0) -- (0,-2.5) ;
\draw (0,0)  -- (1,1) -- (2,0) -- (3,-1) -- (4,-1) -- (5,0) ;
\draw[fill](1,1) circle(.1);
\draw[fill](2,0) circle(.1);
\draw[fill](3,-1) circle(.1);
\draw[fill](4,-1) circle(.1);
\draw[fill](5,0) circle(.1);
\end{tikzpicture}\\[-2.5pt]
&\begin{tikzpicture}[scale=.265]
\draw[thin,<->] (5.3,0) -- (0,0) -- (0,2.5) ;
\draw[thin,->] (0,0) -- (0,-2.5) ;
\draw (0,0)  -- (1,1) -- (2,0) ;
\draw[red, very thick, densely dashed] (2,0) -- (3,1) -- (4,1) ;
\draw (4,1) -- (5,0) ;
\draw[fill](1,1) circle(.1);
\draw[fill](2,0) circle(.1);
\draw[fill](3,1) circle(.1);
\draw[fill](4,1) circle(.1);
\draw[fill](5,0) circle(.1);
\end{tikzpicture}
\,\begin{tikzpicture}[scale=.265]
\draw[thin,<->] (5.3,0) -- (0,0) -- (0,2.5) ;
\draw[thin,->] (0,0) -- (0,-2.5) ;
\draw (0,0)  -- (1,1) -- (2,2) -- (3,1) -- (4,1) -- (5,0) ;
\draw[fill](1,1) circle(.1);
\draw[fill](2,2) circle(.1);
\draw[fill](3,1) circle(.1);
\draw[fill](4,1) circle(.1);
\draw[fill](5,0) circle(.1);
\end{tikzpicture}
\,\begin{tikzpicture}[scale=.265]
\draw[thin,<-] (5.3,0) -- (3,0) ;
\draw[thin,->] (2,0) -- (0,0) -- (0,2.5) ;
\draw[thin,->] (0,0) -- (0,-2.5) ;
\draw (0,0)  -- (1,-1) ;
\draw[red, very thick, densely dashed] (1,-1) -- (2,0) -- (3,0) ;
\draw (3,0) -- (4,1) -- (5,0) ;
\draw[fill](1,-1) circle(.1);
\draw[fill](2,0) circle(.1);
\draw[fill](3,0) circle(.1);
\draw[fill](4,1) circle(.1);
\draw[fill](5,0) circle(.1);
\end{tikzpicture}
\,\begin{tikzpicture}[scale=.265]
\draw[thin,<->] (5.3,0) -- (0,0) -- (0,2.5) ;
\draw[thin,->] (0,0) -- (0,-2.5) ;
\draw (0,0)  -- (1,1) -- (2,0) -- (3,0) -- (4,-1) -- (5,0) ;
\draw[fill](1,1) circle(.1);
\draw[fill](2,0) circle(.1);
\draw[fill](3,0) circle(.1);
\draw[fill](4,-1) circle(.1);
\draw[fill](5,0) circle(.1);
\end{tikzpicture}
\,\begin{tikzpicture}[scale=.265]
\draw[thin,<->] (5.3,0) -- (0,0) -- (0,2.5) ;
\draw[thin,->] (0,0) -- (0,-2.5) ;
\draw (0,0)  -- (1,1) -- (2,0) -- (3,0) -- (4,1) -- (5,0) ;
\draw[fill](1,1) circle(.1);
\draw[fill](2,0) circle(.1);
\draw[fill](3,0) circle(.1);
\draw[fill](4,1) circle(.1);
\draw[fill](5,0) circle(.1);
\end{tikzpicture}
\,\begin{tikzpicture}[scale=.265]
\draw[thin,<->] (5.3,0) -- (0,0) -- (0,2.5) ;
\draw[thin,->] (0,0) -- (0,-2.5) ;
\draw (0,0)  -- (1,1) ;
\draw[red, very thick, densely dashed] (1,1) -- (2,2) -- (3,2) ;
\draw (3,2) -- (4,1) -- (5,0) ;
\draw[fill](1,1) circle(.1);
\draw[fill](2,2) circle(.1);
\draw[fill](3,2) circle(.1);
\draw[fill](4,1) circle(.1);
\draw[fill](5,0) circle(.1);
\end{tikzpicture}
\,\begin{tikzpicture}[scale=.265]
\draw[thin,<->] (5.3,0) -- (0,0) -- (0,2.5) ;
\draw[thin,->] (0,0) -- (0,-2.5) ;
\draw (0,0)  -- (1,-1) -- (2,-1) -- (3,0) -- (4,1) -- (5,0) ;
\draw[fill](1,-1) circle(.1);
\draw[fill](2,-1) circle(.1);
\draw[fill](3,0) circle(.1);
\draw[fill](4,1) circle(.1);
\draw[fill](5,0) circle(.1);
\end{tikzpicture}
\,\begin{tikzpicture}[scale=.265]
\draw[thin,<->] (5.3,0) -- (0,0) -- (0,2.5) ;
\draw[thin,->] (0,0) -- (0,-2.5) ;
\draw[red, very thick, densely dashed]  (0,0)  -- (1,1) -- (2,1) ;
\draw (2,1) -- (3,0) -- (4,-1) -- (5,0) ;
\draw[fill](1,1) circle(.1);
\draw[fill](2,1) circle(.1);
\draw[fill](3,0) circle(.1);
\draw[fill](4,-1) circle(.1);
\draw[fill](5,0) circle(.1);
\end{tikzpicture}\\[-2.5pt]
&\begin{tikzpicture}[scale=.265]
\draw[thin,<->] (5.3,0) -- (0,0) -- (0,2.5) ;
\draw[thin,->] (0,0) -- (0,-2.5) ;
\draw[red, very thick, densely dashed] (0,0)  -- (1,1) -- (2,1) ;
\draw (2,1) -- (3,0) -- (4,1) -- (5,0) ;
\draw[fill](1,1) circle(.1);
\draw[fill](2,1) circle(.1);
\draw[fill](3,0) circle(.1);
\draw[fill](4,1) circle(.1);
\draw[fill](5,0) circle(.1);
\end{tikzpicture}
\,\begin{tikzpicture}[scale=.265]
\draw[thin,<->] (5.3,0) -- (0,0) -- (0,2.5) ;
\draw[thin,->] (0,0) -- (0,-2.5) ;
\draw[red, very thick, densely dashed] (0,0)  -- (1,1) -- (2,1) ;
\draw (2,1) -- (3,2) -- (4,1) -- (5,0) ;
\draw[fill](1,1) circle(.1);
\draw[fill](2,1) circle(.1);
\draw[fill](3,2) circle(.1);
\draw[fill](4,1) circle(.1);
\draw[fill](5,0) circle(.1);
\end{tikzpicture}
\,\begin{tikzpicture}[scale=.265]
\draw[thin,<->] (5.3,0) -- (0,0) -- (0,2.5) ;
\draw[thin,->] (0,0) -- (0,-2.5) ;
\draw (0,0)  -- (1,0) -- (2,-1) -- (3,0) -- (4,1) -- (5,0) ;
\draw[fill](1,0) circle(.1);
\draw[fill](2,-1) circle(.1);
\draw[fill](3,0) circle(.1);
\draw[fill](4,1) circle(.1);
\draw[fill](5,0) circle(.1);
\end{tikzpicture}
\,\begin{tikzpicture}[scale=.265]
\draw[thin,<->] (5.3,0) -- (0,0) -- (0,2.5) ;
\draw[thin,->] (0,0) -- (0,-2.5) ;
\draw (0,0)  -- (1,0) -- (2,1) -- (3,0) -- (4,-1) -- (5,0) ;
\draw[fill](1,0) circle(.1);
\draw[fill](2,1) circle(.1);
\draw[fill](3,0) circle(.1);
\draw[fill](4,-1) circle(.1);
\draw[fill](5,0) circle(.1);
\end{tikzpicture}
\,\begin{tikzpicture}[scale=.265]
\draw[thin,<->] (5.3,0) -- (0,0) -- (0,2.5) ;
\draw[thin,->] (0,0) -- (0,-2.5) ;
\draw (0,0)  -- (1,0) -- (2,1) -- (3,0) -- (4,1) -- (5,0) ;
\draw[fill](1,0) circle(.1);
\draw[fill](2,1) circle(.1);
\draw[fill](3,0) circle(.1);
\draw[fill](4,1) circle(.1);
\draw[fill](5,0) circle(.1);
\end{tikzpicture}
\,\begin{tikzpicture}[scale=.265]
\draw[thin,<->] (5.3,0) -- (0,0) -- (0,2.5) ;
\draw[thin,->] (0,0) -- (0,-2.5) ;
\draw (0,0)  -- (1,0) -- (2,1) -- (3,2) -- (4,1) -- (5,0) ;
\draw[fill](1,0) circle(.1);
\draw[fill](2,1) circle(.1);
\draw[fill](3,2) circle(.1);
\draw[fill](4,1) circle(.1);
\draw[fill](5,0) circle(.1);
\end{tikzpicture}
\end{align*}
\caption{Motzkin and big Motzkin paths}
\end{figure}
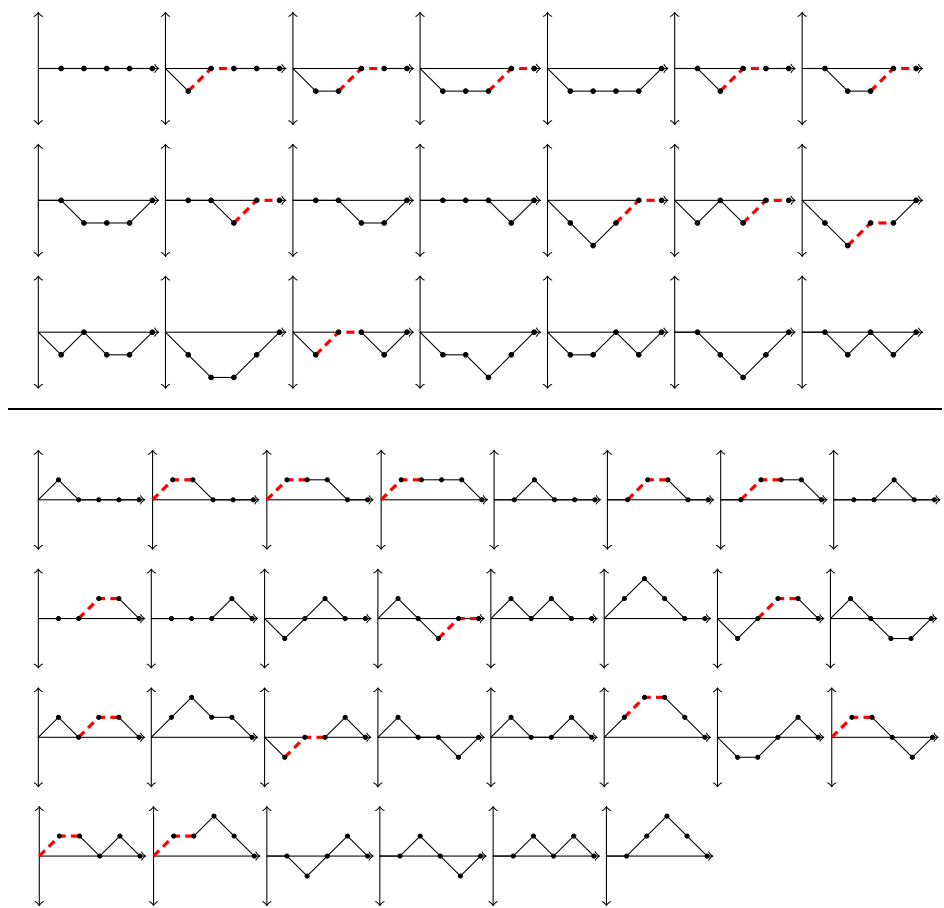

\subsection{Restricted central Delannoy and big Motzkin paths}
We now count paths where given subsequences of paths, namely the subsequences $U\!H$ and $U\!F$, are not allowed. The first ones are, respectively, the \hbox{\emph{UH-avoiding}} big Motzkin paths and the \emph{UH-avoiding} Motzkin paths. In Figure~2 $U\!H$ subsequences and in Figure~3 $U\!F$ subsequences are dashed. Note that these paths are determined by sequences formed only by the $d$ down steps and the $n-2d$ forward steps, and hence its numbers are, respectively, $\sum_{d=0}^n \binom{n-d}{n-2d}\binom{2d}{d}$ and $\sum_{d=0}^n \frac{1}{d+1}\binom{n-d}{n-2d}\binom{2d}{d}$. Thus, UH-avoiding big Motzkin paths and UH-avoiding Motzkin paths, respectively, form the OEIS sequences
\begin{align*}
\text{\href{https://oeis.org/A026569}{A026569}:}&\text{\footnotesize$\Big(\textstyle\sum\limits_{d=0}^n \binom{n-d}{d}\binom{2d}{d}\Big)_{n\geq0}=(1, 1, 3, 5, 13,
\textbf{27}, 67, 153, 375, 893, 2\,189,\dotsc)$}\\
\text{\href{https://oeis.org/A090344}{A090344}:}&\text{\footnotesize$\Big(\textstyle\sum\limits_{d=0}^n \frac{1}{d+1}\binom{n-d}{d}\binom{2d}{d}\Big)_{n\geq0}=(1, 
1, 2, 3, 6, \textbf{11}, 23, 47, 102, 221, 493,\dotsc)$}
\end{align*}

Central Delannoy paths and Schr\"oder paths with $d$ down steps also contain $d$ up steps, but now contain $n-d$ forward steps, and so the sequences are OEIS sequence
\begin{align*}
\text{\href{https://oeis.org/A026375}{A026375}:}&\ \Big(\textstyle\sum\limits_{d=0}^n \binom{n}{d}\binom{2d}{d}\Big)_{n\in\N_0} 
=\text{\small$(1, 3, \textbf{11}, 45, 195, 873, 3\,989, 18\,483,\dotsc)$}
\shortintertext{and OEIS sequence}
\text{\href{https://oeis.org/A007317}{A007317}:}&\ \Big(\textstyle\sum\limits_{d=0}^n \frac{1}{d+1}\binom{n}{d}\binom{2d}{d}\Big)_{n\in\N_0}
=\text{\small$(1, 2, \textbf{5}, 15, 51, 188, 731, 2\,950, \dotsc)$}
\end{align*}

We note that, for every lattice path sequence $V=\big(v_n\big)_{n\in\N_0}$ defined above, we have found a double infinite triangular array $T=\big(b_{n,d}\big)_{0\leq d\leq n}$ such that
\[v_n=\sum_{d=0}^n b_{n,d}\,u_d\quad \text{for every $n\in\N_0$}\]
or, in other words, such that $V^T=T\cdot U^T$, where $U=\big(u_n\big)_{n\in\N_0}$ is either the central lattice path sequence or the Dyck path sequence. In fact, by Lemma~\ref{lemma.MV}, below, in all cases, for two given generating functions $f(x)$ and $g(x)$,
\begin{align*}
&b_{n,d}=[x^n]\big(f(x)(x\,g(x))^d\big)\quad\text{for $0\leq d\leq n$}\,,
\shortintertext{or, equivalently,}
&\sum_{n\geq d}b_{n,d}\, x^n=f(x)(x\,g(x))^d\,.
\end{align*} 
We shorten notations by writing the \emph{Riordan array} $(f(x)\mid g(x))$ for $T=\big(b_{n,d}\big)_{0\leq d\leq n}$. We can now obtain the generating functions of these sequences by adequately transforming \eqref{eq.fgCLP} and \eqref{eq.fgCatalan}. In fact, if $B(x)=\sum_{n\geq0}v_n x^n$ and $A(x)=\sum_{n\geq0}u_n x^n$, then (see~\cite{Riordan})
\begin{align}
B(x)
&=\sum_{n\geq0}\Big(\text{\footnotesize$\sum_{d=0}^n$} b_{n,d}\, u_d\Big)\,x^n\nonumber\\
&=\sum_{d\geq0} u_d\Big(\text{\footnotesize$\sum_{n\geq d}$}b_{n,d}\,x^n\Big)\nonumber\\
&=\sum_{d\geq0}u_d\,f(x)\big(x\,g(x)\big)^d\nonumber\\
&=f(x)\Big(\text{\footnotesize$\sum_{d\geq0}$}u_d\big(x\,g(x)\big)^d\Big)\,,\nonumber\\
\shortintertext{that is,}
B(x)&=f(x)A\big(x\,g(x)\big)\,.\label{eq.riordan}
\end{align} 

We note that, in the case of the restricted central Delannoy and big Motzkin paths, this unifies and clarifies 
entries \href{https://oeis.org/A026569}{A026569}, \href{https://oeis.org/A090344}{A090344},
\href{https://oeis.org/A026375}{A026375}, and \href{https://oeis.org/A007317}{A007317} of the OEIS (Cf.~\cite{Roitner,Yan}), and generally simplifies the corresponding generating functions.

\begin{lemma}\label{lemma.MV}
\begin{align*}
&\text{\small$\left(\frac{1}{1-x}\,\middle|\, \frac{1}{(1-x)^2}\right)=\tri{n+d}{n-d}$}\quad;\\
&\text{\small$\left(\frac{1}{1-x}\,\middle|\, \frac{x}{(1-x)^2}\right)=\tri{n}{2d}$}\quad;\\ 
&\text{\small$\left(\frac{1}{1-x}\,\middle|\, \frac{x}{1-x}\right)=\tri{n-d}{d}$}\quad;\\ 
&\text{\small$\left(\frac{1}{1-x}\,\middle|\, \frac{1}{1-x}\right)=\tri{n}{d}$}\quad.
\end{align*}
\end{lemma}
\begin{proof}
Note that, by the generalized Newton binomial theorem again, if $\alpha=-n$, $n\in\N$, since
$\binom{\alpha}{k}=\frac{(-n)(-n-1)\dotsb (-n-k+1)}{k!}=(-1)^k\binom{n+k-1}{k}$,
\[(1-x)^{-n}=\sum_{k\geq0}\binom{n+k-1}{k}\,x^k\,.\]

Let us prove the first identity. For $f(x)=\frac{1}{1-x}$ and $g(x)=\frac{1}{(1-x)^2}$,
\begin{align*}
[x^n]\big(f(x)(x\,g(x))^d\big) &=[x^n]\left(\text{\footnotesize$\displaystyle\frac{x^d}{(1-x)^{2d+1}}$}\right)\\
&=[x^{n-d}]\left(\text{\footnotesize$\frac{1}{(1-x)^{2d+1}}$}\right)\\
&=[x^{n-d}]\text{\footnotesize$\sum_{m\geq0}\binom{2d+m}{m}\,x^m$}\\
&=\text{\footnotesize$\binom{2d+n-d}{n-d}$}\\
&=\text{\footnotesize$\binom{n+d}{n-d}$}.
\end{align*}
The other identities are proven similarly, being 
\begin{align*}
&[x^n]\left(\text{\footnotesize$\displaystyle\frac{1}{1-x}\left(\frac{x^2}{(1-x)^2} \right)^d$}\right)=
[x^{n-2d}]\text{\footnotesize$\left(\displaystyle\sum_{m\geq0}\binom{2d+m}{m}\,x^m\right)$}\,,\\
&[x^n]\left(\text{\footnotesize$\displaystyle\frac{1}{1-x}\left(\frac{x^2}{1-x} \right)^d$}\right)=
[x^{n-2d}]\text{\footnotesize$\left(\displaystyle\sum_{m\geq0}\binom{d+m}{m}\,x^m\right)$}\,,\\
&[x^n]\left(\text{\footnotesize$\displaystyle\frac{1}{1-x}\left(\frac{x}{1-x} \right)^d$}\right)=
[x^{n-d}]\text{\footnotesize$\left(\displaystyle\sum_{m\geq0}\binom{d+m}{m}\,x^m\right)$}\,.\qedhere
\end{align*}
\end{proof}

Then, the generating function of the sequence of the central Delannoy numbers is, by \eqref{eq.riordan},
\begin{align*}
&\text{\footnotesize$\frac{1}{1-x}\ \frac{1}{\sqrt{1-4\frac{x}{(1-x)^2}}}$}=\frac{1}{\sqrt{1-6 x+x^2}}
\shortintertext{whereas the generating function of the sequence of the central Schr\"oder numbers is}
&\text{\footnotesize$\frac{1}{1-x}\ \frac{2}{1+\sqrt{1-4\frac{x}{(1-x)^2}}}$}= \frac{2}{1-x+\sqrt{1-6 x+x^2}}\,.
\shortintertext{The generating functions of the sequences of 
UF-avoiding central Delannoy numbers and of UF-avoiding Schr\"oder numbers are, respectively}
&\text{\footnotesize$\frac{1}{1-x}\ \frac{1}{\sqrt{1-4\frac{x}{1-x}}}$}= \frac{1}{\sqrt{1-6x+5x^2}}\\
\shortintertext{and}
&\text{\footnotesize$\frac{2}{(1-x)\left(1+\sqrt{1-4\frac{x}{1-x}}\right)}$} = \frac{2}{1-x+\sqrt{1-6x+5x^2}}\,.
\end{align*}

Likewise, the generating function of the sequence of the big Motzkin numbers is
\begin{align*}
&\text{\footnotesize$\frac{1}{1-x}\ \frac{1}{\sqrt{1-4\frac{x^2}{(1-x)^2}}}$} = \frac{1}{\sqrt{1-2x-3 x^2}}
\shortintertext{the generating function of the sequence of the Motzkin numbers is}
&\text{\footnotesize$\frac{2}{(1-x)\left(1+\sqrt{1-4\frac{x^2}{(1-x)^2}}\right)}$}=\frac{2}{1-x+\sqrt{1-2x-3 x^2}}
\end{align*}
and the generating function of the sequence of 
UH-avoiding big Motzkin and Motzkin numbers are, respectively
\begin{align*}
&\text{\footnotesize$\frac{1}{1-x}\ \frac{1}{\sqrt{1-4\frac{x^2}{1-x}}}= \frac{1}{\sqrt{1-x}\sqrt{1-x-4x^2}}$}\\
&\hphantom{\text{\footnotesize$\frac{1}{1-x}\ \frac{1}{\sqrt{1-4\frac{x^2}{1-x}}}$}}= \frac{1}{\sqrt{1 - 2x - 3x^2 + 4x^3}}\\
\shortintertext{and}
&\text{\footnotesize$\frac{2}{(1-x)\left(1+\sqrt{1-4\frac{x^2}{1-x}}\right)}$} = \frac{2}{1-x+\sqrt{1 - 2x - 3x^2 + 4x^3}}\,.
\end{align*}

\end{document}